\documentclass[12pt]{article}
\usepackage{latexsym,amssymb,amsmath,amsfonts,amsthm}
\usepackage{bm}
\usepackage{marginnote}
\usepackage{color}
\newtheorem{theorem}{Theorem}

\newtheorem{lemma}{Lemma}
\newtheorem{remark}{Remark}
\newtheorem{prop}{Proposition}

\def\beq{ \begin{equation} }
\def\eeq{ \end{equation} }
\def\mn{\medskip\noindent}

\def\ep{\epsilon}

\def\square{\vcenter{\vbox{\hrule height .4pt
  \hbox{\vrule width .4pt height 5pt \kern 5pt
        \vrule width .4pt} \hrule height .4pt}}}

\def\TT{\mathbb{T}}
\def\ZZ{\mathbb{Z}}
\def\FF{\mathcal{F}}

\def\LL{\mathcal{L}}
\def\varep{\varepsilon}

\def\hh{\hspace{1ex}}
\def\ep{\varepsilon}
\def\GW{\textbf{GW}}
\def\hh{\hspace{1ex}}
\definecolor{darkblue}{rgb}{0,0,0.6}

\definecolor{darkgreen}{rgb}{0,0.7,0}

\begin{document}

\title{Exponential growth and continuous phase transitions for the contact process on trees}
\author{Xiangying Huang}

\date{\today}	

\maketitle

\begin{abstract}
We study the supercritical contact process on Galton-Watson trees and periodic trees. We prove that if the contact process survives weakly then  
it dominates a supercritical Crump-Mode-Jagers branching process. Hence the number of infected sites grows exponentially fast. As a consequence we conclude that the contact process dies out at the critical value $\lambda_1$ for weak survival, and the survival probability $p(\lambda)$ is continuous with respect to the infection rate $\lambda$. Applying this fact, we show the contact process on a general periodic tree experiences two phase transitions in the sense that $\lambda_1<\lambda_2$, which confirms a conjecture of Stacey's \cite{Stacey}. We also prove that if the contact process survives strongly at $\lambda$ then it survives strongly at a $\lambda'<\lambda$, which implies that the process does not survive strongly at the critical value $\lambda_2$ for strong survival.
\end{abstract}

\section{Introduction}
Harris \cite{Harris} introduced the contact process on $\ZZ^d$ in 1974, which has been extensively studied since then. The contact process on a graph is usually viewed as a model that describes the spread of an infection. Vertices of the graph represent individuals and the states 0 and 1 indicate that an individual is healthy or infected. The contact process can be defined on any graph as follows: infected sites become healthy at rate 1, while healthy sites become infected at rate $\lambda$ times the number of infected neighbors. 

Pemantle \cite{Pemantle} began the study of the contact process on trees and found that there exist two critical values 
\begin{align*}
\lambda_1&=\inf\{ \lambda: P(\xi^0_t\neq \varnothing \text{ for all }t)>0\}\\
\lambda_2&=\inf \{\lambda: \liminf_{t\to\infty}P(0\in \xi^0_t)>0\},
\end{align*}
where $\xi^0_t$ denotes the contact process on the tree starting from only the root infected. The contact process is said to survive \textit{weakly} if the process survives but the root $0$ is infected for finitely many times almost surely, and survive \textit{strongly} if the root $0$ is infected for infinitely many times with positive probability. It is natural to guess that the contact process dies out at $\lambda_1$ and does not survive strongly at $\lambda_2$. This is true for the contact process on a $d$-regular tree $\TT^d$ where each vertex has degree $d+1$. (See Liggett \cite{Liggett} for most of the known results on regular trees.) However, proofs on regular trees rely heavily on translation invariance, and special functions such as $w_\rho(\xi_t)=\sum_{x\in\xi_t}\rho^{\ell(x)}$, where $\rho$ is some positive constant and $\ell(x)$ is a function from $\TT^d$ to $\ZZ$ so that for each $x\in \TT^d$, $\ell(y)=\ell(x)-1$ for exactly one neighbor $y$ of $x$ and $\ell(y)=\ell(x)+1$ for the other $d$ neighbors $y$ of $x$. The expectation of $w_\rho(\xi_t)$ satisfies a submultiplicative relation
$$Ew_\rho(\xi_{t+s})\leq E w_\rho(\xi_t)Ew_\rho(\xi_s)$$
which implies $Ew_\rho(\xi_t)\geq [\phi(\rho)]^t$ where $\phi(\rho)=\lim_{t\to\infty} [E w_\rho(\xi_t)]^{1/t}$. Properties of $\phi(\rho)$ help us obtain a lot of detailed information about the behavior of the contact process on regular trees.

\subsection{Galton-Watson trees}
Turning to Galton-Watson trees, the analysis of the contact process becomes more complicated because of the spatial heterogeneity and randomness in the tree structure. Throughout the discussion we consider the contact process on Galton-Watson trees with an offspring distribution $D$ that is a random variable on $\mathbb{N}$ that satisfies 
\begin{equation}\label{offdist}
P(D\geq 1)=1 \text{ and } ED>1.
\end{equation}

A Galton-Watson tree is said to have a \textit{subexponential} offspring distribution if its offspring distribution $D$ satisfies 
$$\limsup_{k\to\infty}  \left(\log P(D=k)\right)/k=0.$$
In this case it is proved in Huang and Durrett \cite{HD18} that 
\begin{theorem} \label{subexp}
If the offspring distribution $D$ for a Galton-Watson tree is subexponential and has mean $\mu>1$ then $\lambda_1=\lambda_2=0$.
\end{theorem}
\mn
Clearly when the offspring distribution is subexponential the contact process dies out at both critical values. To give some intuition why $\lambda_2=0$ we discuss briefly the case where the offspring distribution satisfies $P(D\geq k)=\exp(-k^{\alpha})$ for all $k\geq 0$ where $\alpha<1$. Within distance $k$ of the root there are roughly $\mu^k$ vertices where $\mu$ is the mean offspring number. Hence the vertex with the largest degree within distance $k$ has degree $\approx (k\log \mu)^{1/\alpha}$. A vertex with degree $n$ along with its neighbors is called a star graph with degree $n$. The contact process on a star graph with degree $n$ can survive for time $\exp( C\lambda^2n)$ for some $C>0$ when $n$ is sufficiently large (see \cite{CD09}). Here and in what follows $C$ is a positive constant whose value can change.
While surviving on a star graph within distance $k$ of the root, the contact process can try to push the infection back to the root with success probability $(\lambda/(1+\lambda))^k$ for each attempt. 
Since the process can survive for time  $\exp(C\lambda^2 (k\log \mu)^{1/\alpha})$ on a star graph with degree $(k\log \mu)^{1/\alpha}$, the probability that it fails to push the infection back to the root is (while omitting some details)
\beq\label{subpush}
\leq \left(1- (\lambda/(1+\lambda))^k\right)^{\exp(C\lambda^2(k\log \mu)^{1/\alpha})}\leq \exp\left( -\left(\frac{\lambda}{1+\lambda}\right)^k  \exp(C\lambda^2(k\log \mu)^{1/\alpha})\right).
\eeq
Since the first term in the exponent decays exponentially while the second grows superexponentially, the failure probability goes to 0 as $k\to\infty$ for any positive values of $\lambda$. That is, the contact process can successfully push the infection back to the root and thus survive strongly for any $\lambda>0$.

It turns out that Theorem \ref{subexp} is sharp. The critical value is positive on Galton-Watson trees with tails thinner than the subexponential distribution. Recently Shankar Bhamidi, Danny Nam, Oanh Nguyen and Allan Sly \cite{Sly19} proved that 
\begin{theorem}
Consider the contact process on the Galton-Watson tree with offspring distribution
$D$. If $E(\exp(cD))<\infty$
for some $c > 0$, then $\lambda_1>0$.
\end{theorem}
\mn
An offspring distribution with this property is said to have an \textit{exponential tail}, which is thinner than the tail of the subexponential distribution. For example, suppose $P(D\geq k)=\exp(-ck)$ for $k\geq 0$ and some $c>0$. Following the previous heuristics the vertex with the largest degree within distance $k$ of the root has degree $\approx (k\log \mu)/c$. Again the contact process can survival for time $\exp(C\lambda^2 k)$ on a star graph with degree $(k\log \mu)/c$ for some $C>0$, and can push the infection back to the root with failure probability 
\beq\label{exppush}
\left(1- (\lambda/(1+\lambda))^k\right)^{\exp(C\lambda^2k)}\approx \exp\left( -\left(\frac{\lambda}{1+\lambda}\right)^k  \exp(C\lambda^2k)\right)
\eeq
when $\lambda$ is sufficiently small. As $\log(\lambda/(1+\lambda))+C\lambda^2<0$ when $\lambda$ is small, (\ref{exppush}) does not go to 0 when $k$ goes to infinity. This provides some intuition why $\lambda_1$ is strictly positive when the offspring distribution has an exponential tail. Comparing (\ref{subpush}) and (\ref{exppush}), we can see the difference between the subexpnential distribution and the exponential-tail distribution is that in the first case the largest degree within distance $k$ ( i.e., $(k\log \mu)^{1/\alpha}$) grows superlinearly in $k$ and overwhelms the success probability for pushing the infection back to the root. 

Since $\lambda_2\geq \lambda_1>0$ on Galton-Watson trees with exponential tails, it is now interesting to (i) determine if $\lambda_1<\lambda_2$ and (ii) understand the behavior of the contact process at the critical values $\lambda_1$ and $\lambda_2$. We have not been able to solve the first difficult problem but we have solved the second.

The first step is to construct a supercritical Crump-Mode-Jagers (CMJ) branching process that is dominated by the contact process. As a consequence, we prove that the number of infected sites in the contact process grows exponentially fast on the event of survival. Let $\xi^0_t$ be the contact process on $\GW(D)$ with only the root infected initially. 
\begin{theorem}\label{main}
Suppose the offspring distribution $D$  satisfies (\ref{offdist}). When $\lambda>\lambda_1(\GW(D))$, there is a positive constant $c$ so that
\begin{equation}\label{expg}
\liminf_{t\to\infty} |\xi^0_t|/e^{c t} >0 \quad \text{ $P_\lambda$-a.s.  on }\Omega_\infty \equiv\{\xi^0_t\neq \varnothing \hh \forall t\geq 0\}.
\end{equation}
\end{theorem}
\mn
Here $P_\lambda$ is the annealed measure which is defined rigorously in Section 2. Suppose an individual in the comparison CMJ branching process gives birth only once and let $\tau$ denote that birth time. We can then look at the branching process in terms of time blocks of duration $\tau$ and roughly know the growth of this process by time $\tau$. This suggests that a ``block argument" should be a good strategy to prove the continuity of the phase transition at $\lambda_1$. Indeed, if the contact process $\xi^{\lambda}_t$ with infection rate $\lambda$ survives then it dominates a supercritical CMJ branching process. For a sufficiently small $\ep>0$, the contact process $\xi^{\lambda-\ep}_t$ should behave very much the same as $\xi^{\lambda}_t$ within a finite time block of length $\tau$ and thus also dominate a supercritical branching process. Hence we can show that
\begin{theorem}\label{cont}
The contact process on $\GW(D)$ dies out at  $\lambda_1(\GW(D))$ and the survival probability $p(\lambda)$ is continuous in $[0,\infty)$.
\end{theorem}
\mn

Finally, we prove that
\begin{theorem}\label{lambda2}
The contact process on $\GW(D)$ does not survive strongly at $\lambda_2(\GW(D))$.
\end{theorem}
\mn
The proof of this result is more complicated than that of Theorem \ref{cont} but similar in spirit. We find the suitable ``block events" and construct a comparison process based on the block events. If the contact process with infection rate $\lambda$ survives strongly then it dominates a comparison process that guarantees strong survival, and hence the contact process survives strongly at $\lambda-\ep$ for  sufficiently small $\ep>0$.

\subsection{Periodic trees}
Let $\TT_{\kappa}=(a_1, a_2, \dots, a_\kappa)$ be a periodic tree in which the number of children in successive generations is $a_1, a_2, \ldots, a_\kappa$. We only assume $a_i\geq 1$ for $1\leq i\leq \kappa$ and $a_1\cdot a_2\cdots  a_k\neq 1$. In order for the graph to be quasi-transitive and thus simplify the computation, $\TT_\kappa=(a_1,a_2,\dots,a_\kappa)$ is arranged in the way that every vertex $x$ in $\TT_\kappa$ has one neighbor above it and $d(x)$ neighbors below it, where $d(x)\in\{a_1,\dots,a_\kappa\}$ is the offspring number of $x$. Without loss of generality, a distinguished vertex $o$ is chosen to be the root where $d(o)$ can be any value in $\{a_1, a_2, \ldots, a_\kappa\}$.

The treatment for the contact process on Galton-Watson trees can be extended to prove analogous conclusions on periodic trees.
\begin{theorem}\label{pgrowth}
Let $\xi^o_t$ denote the contact process on $\TT_\kappa$ starting with vertex $o$ infected.\\
\mn
(i) When $\lambda>\lambda_1(\TT_\kappa)$, there is a positive constant $c$ so that
\begin{equation}
\liminf_{t\to\infty} |\xi^o_t|/e^{c t} >0 \quad \text{ $P_\lambda$-a.s.  on }\Omega_\infty \equiv\{\xi^o_t\neq \varnothing \hh \forall t\geq 0\}.
\end{equation}
\mn
(ii) The contact process $\xi^o_t$ on the periodic tree $T_\kappa$ dies out at $\lambda_1(\TT_\kappa)$ and does not survive strongly at $\lambda_2(\TT_\kappa)$.
\end{theorem}

We are also interested in determining if $\lambda_1<\lambda_2$ on general periodic trees, which is believed to be true by Pemantle \cite{Pemantle} and Stacey \cite{Stacey}. Stacey was able to prove $\lambda_1<\lambda_2$ on what he called ``isotropic block trees", which are to some extent similar to periodic trees but have more stringent assumptions on the structure. Intuitively, the assumptions require that there exist automorphisms on the tree that map distinguished vertices to one another. The rigorous definition of the isotropic block trees is a bit involved and hence will not be stated here. Curious readers are referred to page 1718 and 1719 in \cite{Stacey} for a full account. 

%

%
%

To see the connection between isotropic block trees and periodic trees we will  present two examples of the former.

\mn
\textbf{Example 1.} $G=(a_1, a_2, \dots, a_\kappa)$ where $a_i=1$ for $1\leq i\leq \kappa-1$ and $a_\kappa\geq 2$.

\mn
\textbf{Example 2.} $G=(a_1,a_2)$ where $a_1a_2\neq 1$, which is a general period-2 tree.\\

The above examples suggest that the isotropic block trees are, to some extent, similar to periodic trees. However, due to the somewhat stringent assumptions on the tree structure there are still many examples of periodic trees that fall outside this category.  One of the simplest examples among them is the periodic (2,3,4) tree, where the the number of children in successive generations is 2,3,4. 

The special structure of the isotropic blocks is a key ingredient that Stacey uses in his proof of
\beq\label{isotree}
\exp(\alpha t)\leq E(|\xi^A_t|)\leq C\exp(\alpha t),
\eeq
where $C$ is a constant, $\alpha$ is a continuous function of $\lambda$ and $A$ is a set of vertices chosen specially, see Proposition 2.1 in \cite{Stacey}. For example, for a period-2 tree $(a_1,a_2)$ the initially infected set $A$ is chosen to be a finite tree of depth 2 where the first generation has $a_1$ offspring and each vertex in the second generation has $a_2$ offspring. The relation (\ref{isotree}) implies that $\alpha=0$ when $\lambda=\lambda_1$. Hence at $\lambda_1$ not only $E(|\xi^A_t|)\leq C$ for all $t\geq 0$ but also $\xi^A_t$ dies out on an isotropic block tree. These facts are the key elements in Stacey's proof. 

As we will see, the extinction at $\lambda_1$ is the real key to the proof of the existence of an intermediate phase. For this reason we can prove $\lambda_1<\lambda_2$ for general periodic trees without trying to obtain the relation in (\ref{isotree}), which is a nice relation but stronger than necessary for our purpose. We confirm Stacey's conjecture by showing

\begin{theorem}\label{intphase}
The contact process on the periodic tree $T_\kappa$ has an intermediate phase in the sense that  $\lambda_1<\lambda_2$.
\end{theorem}

\vspace{0.5cm}
The outline of the rest of the paper is as follows. In Section 2, we will give an explicit construction of the contact process on Galton-Watson trees. Theorems 3, 4 and 5 will be proved in Section 3, 4 and 5, respectively. Periodic trees will be discussed in Section 6.

\section{Model definition}\label{model}
The set of all realizations of the Galton-Watson trees with offspring distribution $D$ is denoted by $\GW(D)$ while a certain realization is denoted by $\TT\in \GW(D)$. Condition (\ref{offdist}) guarantees that $\TT$ is infinite almost surely.
 Let $P_{\TT,\lambda}$ be the law of the contact process with infection rate $\lambda$ on the tree $\TT$, and let $\mu$ be the probability measure over the realizations in $\GW(D)$. We define the annealed measure $P_\lambda$ to be
$$P_\lambda(\hh\cdot\hh)=E_{\TT \sim \mu}(P_{\TT,\lambda}(\hh\cdot\hh))=\sum_{\TT \in \GW(D)}P_{\TT,\lambda}(\hh\cdot\hh)\mu(\TT).$$
To simplify notation we will drop the index $\lambda$ when the context is clear and use $\GW(D)$ to also represent a Galton-Watson tree with offspring distribution $D$.

Let $\xi^0_t$ denote the contact process on $\GW(D)$ with initially only the root infected. The critical value for weak survival for the contact process on $\TT$ is defined as
$$\lambda_1(\TT)=\inf\{ \lambda: P_{\TT,\lambda}(\xi^0_t\neq\varnothing \text{ for all }t\geq 0)>0 \}.$$
In fact, $\lambda_1(\TT)$ does not depend on the specific realization $\TT\in\GW(D)$. That is, $\lambda_1(\TT)$ is a constant $\mu$-almost surely. Pemantle  proved a slightly different version of this result in Proposition 3.1 in \cite{Pemantle}. We will follow his proof to show 
\begin{prop}\label{constant}
$\lambda_1(\TT)$ is a constant $\mu$-almost surely.
\end{prop}
\begin{proof}
Let $q(\TT)=P_{\TT,\lambda}(\xi^0_t\neq\varnothing \hh \forall t\geq 0)$ and $p_1=P_\lambda(q=0)$. If the root has $n$ children then there are $n$ independent identically distributed subtrees that are themselves Galton-Watson trees with offspring distribution $D$, which we denote by $\TT_1,\dots,\TT_n$. If $q(\TT_i)>0$ for some $1\leq i\leq n$ then $q(\TT)\geq \frac{\lambda}{1+\lambda}q(\TT_i)>0$, where the first term is the probability to push the infection to the root of $\TT_i$. It follows that if the contact process dies out on $\TT$ then it has to die out on all $n$ subtrees $\TT_1,\dots,\TT_n$.

Let $f$ be the generating function of $D$ and let $a_k=P(D=k)$. By the independence of the structures of the subtrees,
\beq\label{fixp}
p_1=P_\lambda(q=0)\leq \sum_{n}a_n(p_1)^n=f(p_1)
\eeq
Note that $f'(1)=ED>1$ by (\ref{offdist}). (\ref{fixp}) combined with $f'(1)>1$ and $f(0)=0$ implies either $p_1=1$ or $p_1=0$. Hence $\lambda_1(\TT)$ is a constant $\mu$-almost surely.
\end{proof}

The critical value for strong survival is defined as 
$$\lambda_2(\TT)=\inf\{ \lambda: P_{\TT,\lambda}(0\in \xi^0_t \hh\text{ infinitely often})>0 \}.$$
It follows from the same reasoning as in the proof of Proposition \ref{constant} that
\begin{prop}
$\lambda_2(\TT)$ is a constant $\mu$-almost surely.
\end{prop}
\mn 
From now on we write $\lambda_1(\GW(D))$ and $\lambda_2(\GW(D))$ for the critical values for the contact process on a Galton Watson tree with offspring distribution $D$.

\section{Exponential growth}\label{expgrowth}
In this section we prove Theorem \ref{main} by showing that the contact process dominates a supercritical Crump-Mode-Jagers branching process. On an inhomogeneous structure such as the Galton-Watson tree, the infection state of a vertex $x$ is correlated with the tree structure near $x$. To address this issue we observe that starting with only the root infected, when a vertex $x$ becomes infected for the first time no information is known about the structure of the subtree rooted at $x$ (i.e., the subtree made of $x$ and all of its descendants). Hence we can ``decouple" the infection state of $x$ and the tree structure below it.

For a vertex $x\in\TT$ we define $S(x)$ to be the \textit{subtree} with root $x$, that is, $S(x)$ contains $x$ together with all the descendants of $x$ on $\TT$. For a finite set $A \subset \TT$, we define its \textit{frontier} $F(A)$ to be the set of points $x\in A$ for which at least one of its children, say $x'$, has $S(x')\cap A=\varnothing$. Let $A'$ denote the set of $x'$ such that $x'$ is the child of some $x\in A$ and $S(x')\cap A=\varnothing$.

Let $A_t=\cup_{s\leq t} \hh\xi^0_s$ be the set of vertices ever infected by time $t$. We are interested in the set of vertices that have never been infected and are accessible to the infection $\xi^0_t$ at time $t$, i.e.,
$$
B_t=\{ x\in (A_t)': \text{ the parent of $x$ is in $\xi^0_t \cap F(A_t)$}\}.
$$
Let $\tau_k=\inf\{t\geq 0: |\xi^0_t\cap F(A_t)|\geq k \}$. The fact that each vertex in $|\xi^0_t\cap F(A_t)|$ corresponds to at least one vertex in $B_t$ gives $|B_{\tau_k}|\geq k$, i.e., there are at least $k$ unexplored subtrees accessible to the contact process $\xi^0_{\tau_k}$. At time $\tau_k$, $\xi^0_t$ can try to start infections on each of the $k$ subtrees, among which each infection could survive forever with a positive probability equal to $P_\lambda(\xi^0_t\neq\varnothing \hh \forall t\geq 0)$. When $k$ is taken to be sufficiently large, at time $\tau_k$ the contact process $\xi^0_t$ gives birth to an expected number of more than one new process that lives forever, which make up an underlying branching process.

Recall that $\Omega_\infty =\{\xi^0_t\neq \varnothing \hh \forall t\geq 0\}$. We will begin the proof of Theorem \ref{main} by showing
\begin{lemma}\label{BPcomparison}
Suppose $\lambda>\lambda_1(\GW(D))$. If for every $k\in\mathbb{N}$, $\tau_k<\infty$ $P_\lambda$-almost surely on $\Omega_\infty$ then \eqref{expg} holds. 
\end{lemma}

\begin{proof}
When $\lambda>\lambda_1(\GW(D))$ we have $\rho:=P_\lambda(\xi^0_t\neq\varnothing \hh \forall t\geq 0)>0$. Observe that with probability $\geq e^{-2}(1-e^{-\lambda})$ an infected vertex  $x\in \xi^0_t\cap F(A_t)$ can infect its child $x'\in B_t$ within time 1 and $x'$ will stay infected until time 1. 

Since the subtree $S(x')$ is still unexplored when $x'\in B_t$ first becomes infected, we can first generate $S(x')$ according to the measure $\mu$ of $\GW(D)$ and then start a contact process on it. The probability that the process $\xi^{x'}_t$ restricted to $S(x')$ will survive is 
$$E_{S(x')\sim \mu}(P_{S(x'),\lambda}(\xi^{x'}_t\neq t \text{ for all }t\geq 0))=\rho.$$

If this infection on $S(x')$ survives, we say $x'$ is a {\it particle with an infinite line of descendants}.  Note that each particle with an infinite line of descendants will contribute at least size 1 to $|\xi^0_t|$ at all times. So it suffices to show the number of particles with an infinite line of descendants grows exponentially in time on the event $\Omega_\infty$.

We will construct a branching process $Z_t$ that is dominated by the set of particles that has an infinite line of descendants. Let $Z_0$ contain only the root 0. Since $\{\tau_k<\infty\}$ almost surely on $\Omega_\infty$ we can choose $M_1$ large so that $P(\tau_k < M_1)\geq\rho/2$. If $\tau_k \geq M_1$ we discard the initial particle and there are no offspring. If $\tau_k < M_1$ then the initial particle gives birth to $X\sim Binomial(k, e^{-2}(1-e^{-\lambda}))$ offspring at time $\tau_k$. This is because the birth events of different vertices in $\xi^0_t\cap F(A_t)$ are independent, since they involve vertices that are disjoint in space. If $k$ is chosen large enough then the number of offspring $\bar X$ that will start a new process on its subtree with $\tau_k < M_1$ has expectation
$$
E\bar X \ge ke^{-2}(1-e^{-\lambda})\rho/2>1,
$$
which implies that $Z_t$ is supercritical.

The branching process is a Crump-Mode-Jagers (CMJ) branching process \cite{CMJ} because the initial particle gives birth at time $\tau_k$. 
If $N(t)$ is the expected number of births by time $t$ then the Malthusian parameter $c$ is defined through
$$
\int_0^\infty e^{-c t} dN(t) = 1.
$$ 
Since our process has no births after time $M_1$
it is easy to see that it satisfies the conditions in Theorem 2 of \cite{Ganuza}, which gives
$$
W(t):=\frac{Z_t}{e^{c t}}\to W \quad \text{ a.s. as }t\to\infty,
$$
where $W$ is a non-negative random variable with $P(W>0)>0$. To improve this to the desired conclusion we note that
for any $m < \infty$, $\tau_m<\infty$ almost surely on $\Omega_\infty$. At time $\tau_m$, each vertex in $\xi^0_t\cap F(A_t)$ can start an independent branching process construction with probability at least $e^{-2}(1-e^{-\lambda})$ and each branching process will live forever with probability $P(W>0)$. As long as one of the branching processes lives forever we will have the desired exponential growth of $\xi^0_t$. Hence 
\begin{align*}
P\left(\liminf_{t\to\infty}|\xi^0_t|/e^{c t} > 0 \right) &\ge P(\Omega_\infty)P(Binomial(m, e^{-2}(1-e^{-\lambda})P(W>0))\geq 1)\\
&= P(\Omega_\infty)\left(1-\left(1-e^{-2}(1-e^{-\lambda})P(W>0)\right)^m\right)
\end{align*}
for any $m\in \mathbb{N}$. Letting $m \to\infty$ and noting that the limit can only be positive on $\Omega_\infty$ proves the desired result.
\end{proof}

Having established Lemma \ref{BPcomparison}, to complete the proof of Theorem \ref{main} it remains to show

\begin{lemma}\label{taukfinite}
For any $k\in\mathbb{N}$, $\tau_k < \infty$ almost surely on $\Omega_\infty$.
\end{lemma}

\begin{proof}
If some particle $x\in \xi^0_t\cap F(A_t)$ gives birth to a particle $y\in B_t$, then at the moment of this birth $y$ is added to $\xi^0_t\cap F(A_t)$. In the worst case $y$ is the last unexplored child of $x$, so $x$ is removed from $\xi^0_t\cap F(A_t)$ and the size of $\xi^0_t\cap F(A_t)$ remains the same. Note that the size of $\xi^0_t\cap F(A_t)$ will not decrease at a birth event.

When $y\in B_t$ becomes infected, we try to grow $|\xi^0_t\cap F(A_t)|$ through consecutive birth events on the subtree $S(y)$. Since $S(y)$ has not been explored by the infection, every vertex in $S(y)$ still has all of its offspring unexplored. When a children of $y$ (except the last one) is infected, the size of $\xi^0_t\cap F(A_t)$ increases by 1. Hence if every particle gives birth onto a particle with offspring number at least 2, we need at most $2k$ consecutive birth events on $S(y)$ for $|\xi^0_t\cap F(A_t)|$ to reach size $k$.

Taking account of the first birth event from $x$ to $y$, we need $2k+1$ consecutive birth events before any death occurs and each of them is onto a vertex with offspring number at least 2.  Since each vertex in $S(y)\cap \xi^0_t\cap F(A_t)$ is connected to at least one vertex outside, the probability that a birth occurs before any death in $S(y)\cap \xi^0_t\cap F(A_t)$ has probability $\geq \lambda/(1+\lambda)$. 

We start the first trial at the first time when $|\xi^0_t\cap F(A_t)|\geq 1$, i.e., when $t=0$. A trial is said to have failed if a death occurs or a birth to a vertex with offspring number 1 occurs before we have $2k+1$ consecutive birth events to vertices with offspring number at least 2. Define the running time $S_i$ of the $i$-th trial to be the first time either this trial fails or succeeds. $S_i$ is dominated by the sum of $2k+1$ i.i.d. $Exp(1+\lambda)$ random variables and hence is finite almost surely. Let $T_0=0$ and define
$$
T_{i}=\inf\{ t\geq T_{i-1}+S_{i-1} : |\xi^0_t\cap F(A_t)|\geq 1\}
$$
to be the time when the $i$-th trial starts. If $T_i=\infty$ for some $i\geq 1$ then set $T_j=\infty$ for all $j\geq i$. Note that with an independent probability 
$$
\geq \left(\frac{\lambda}{1+\lambda}\cdot P(D\geq 2)\right)^{2k+1} \equiv \alpha_k
$$ 
a trial will succeed, so it takes finitely many attempts to have one success, which then gives $|\xi^0_t\cap F(A_t)|\geq k$. If the $i$-th trial fails, either we still have $|\xi^0_t\cap F(A_t)|\geq 1$, which means $T_{i+1}=T_i+S_i$, or $|\xi^0_t\cap F(A_t)|=0$ and we have to wait until $|\xi^0_t\cap F(A_t)|$ becomes 1 before starting the next trial. On the event that $\xi^0_t$ survives, infinitely many sites will be infected, i.e.,
$$
\lim_{t\to\infty} |A_t|=\infty \quad \text{ a.s.  on } \Omega_\infty.
$$
Since $A_t$ cannot grow when $\xi^0_t\cap F(A_t)=\varnothing$ we have $|\xi^0_t\cap F(A_t)|\geq 1$ infinitely often on the event $\Omega_\infty$, meaning that for any $n\in \mathbb{N}$
$$
\{T_n<\infty\} \hh \text{ a.s.  on }\Omega_\infty.
$$ 
Hence we have infinitely many trials almost surely on the event of survival. Since the birth and death symbols involved in different trials come from regions disjoint in space and time, each trial succeeds with probability $\ge \alpha_k$ independent of the fates of previous trials. Therefore we will have a success in finite time almost surely on $\Omega_\infty$. The proof of Lemma \ref{taukfinite} is complete and Theorem \ref{main} has been established. 
\end{proof}

\begin{remark}
Our method is robust in the sense that it works on a broad class of trees. For example, in order to prove exponential growth for the contact process on periodic trees it suffices to extend the proof of Lemma \ref{taukfinite} to periodic trees. The key element in the proof of Lemma \ref{taukfinite} is the expansion of the infected frontier $\xi^0_t\cap F(A_t)$ through consecutive birth events, which holds true for periodic trees too. It is easy to see that if $a_i\geq 2$ for all $1\leq i\leq \kappa$ in the periodic tree $\TT_\kappa=(a_1,\dots, a_\kappa)$, then each birth event in the infected frontier will cause the size $|\xi^0_t\cap F(A_t)|$ to increase by 1. If we only assume $a_1a_2\cdots a_\kappa\neq 1$, then we need $\kappa$ consecutive birth events in the frontier to ensure that the size $|\xi^0_t\cap F(A_t)|$ increases by 1. 
\end{remark}

\section{Continuous survival probability}

In the branching process $Z_t$ in the proof of Lemma \ref{BPcomparison}, a particle can only give birth within time $M_1$, which means we are only looking at a finite time ``block" in the corresponding graphical representation. We drop the superscript $0$ and 
let $\xi^{\lambda}_t$ denote the contact process with infection rate $\lambda$ and recovery rate 1 starting from the root infected. Within this ``block" the process $\xi^{\lambda}_t$ and $\xi^{\lambda'}_t$ should behave very much the same if $\lambda'$ is sufficiently close to $\lambda$. If $\xi^{\lambda}_t$ survives with positive probability, tuning down the infection rate $\lambda$ slightly to $\lambda'$ the resulting process $\xi^{\lambda'}_t$ should still behave similarly in the finite time ``block", which then implies the survival of $\xi^{\lambda'}_t$. We will follow this idea to prove both results in this section. A superscript $\lambda$ is used to denote the infection rate. For example, $\tau^{\lambda}_k$ is the corresponding hitting time of the event $\{|\xi^{\lambda}_t\cap F(A_t)|\geq k\}$.

\begin{lemma}\label{lambda1}
The contact process  on $\GW(D)$ dies out at  $\lambda_1(\GW(D))$.
\end{lemma}

\begin{proof}
To prove the conclusion it suffices to show $P_{\lambda}(\xi^0_t\neq \varnothing \text{ for all }t\geq 0)>0$ implies $\lambda>\lambda_1$. Let $\rho:=P_\lambda(\xi^0_t\neq\varnothing \text{ for all }t\geq 0)>0$. For reasons that will become clear later we choose $k, M_1$ so that 
\beq\label{choice}
ke^{-2}(1-e^{-\lambda})\rho/2>2 \quad \text{and}\quad P(\tau^{\lambda}_k< M_1)\geq \rho/2.
\eeq
For a small $\delta>0$ we have 
$$P(\tau^{\lambda-\delta}_k<M_1)\geq P(\{\xi^{\lambda}_t= \xi^{\lambda-\delta}_t \text{ for all }t\leq \min\{\tau^{\lambda}_k,M_1\}\}\cap \{\tau^{\lambda}_k<M_1\}).$$
For any $\ep>0$, when $\delta$ is sufficiently small we have 
$$P(\xi^{\lambda}_t \neq \xi^{\lambda-\delta}_t \text{ for some }t\leq \min\{\tau^{\lambda}_k,M_1\})<\ep.$$
Hence,
$$P(\tau^{\lambda-\delta}_k < M_1)\geq 1-\ep-(1-\rho/2)=\rho/2-\ep.$$

Following the proof of Lemma \ref{BPcomparison} we can construct a branching process $Z_t$ dominated by $\xi^{\lambda-\delta}_t$, where a particle gives birth to $X\sim Binomial(k, e^{-2}(1-e^{-\lambda-\delta}))$ offspring if $\tau^{\lambda-\delta}_k<M_1$, and to no offspring if $\tau^{\lambda-\delta}_k\geq M_1$. When $\ep$ and $\delta$ are sufficiently small, by the choice of $k$ in (\ref{choice}) the mean number of offspring is
$$EX \geq ke^{-2}(1-e^{-\lambda-\delta})(\rho/2-\ep)>1.$$
Hence $Z_t$ is supercritical and survives with positive probability, which implies $\xi^{\lambda-\delta}_t$ survives with positive probability. So $\lambda>\lambda-\delta \geq \lambda_1$ and this completes the proof.
\end{proof}

Let $p(\lambda):=P_\lambda(\xi^0_t\neq \varnothing \hh \forall t\geq 0)$ be a function of $\lambda$ that represents the survival probability of the contact process $\xi^\lambda_t$.
\begin{lemma}\label{continuity}
$p(\lambda)$ is continuous on $[0,\infty)$.
\end{lemma}
\begin{proof}
The right continuity of $p(\lambda)$ is immediate since 
$$P_\lambda(\xi^0_t\neq \varnothing) \downarrow p(\lambda)$$
as $t\uparrow \infty$ and $P_\lambda(\xi^0_t\neq \varnothing)$ is increasing and continuous in $\lambda$ for each $t$. 

We have proved that $p(\lambda)\equiv 0$ on $[0,\lambda_1]$, so it remains to prove the left continuity for $p(\lambda)$ on $(\lambda_1,\infty)$. Take $\lambda>\lambda_1$, which trivially gives $p(\lambda)>0$. We will show for any $\ep>0$, there exists $\lambda'<\lambda$ such that $p(\lambda')>p(\lambda)-\ep$. 

Since $\lambda>\lambda_1$ there exists $\lambda_0$ such that $\lambda>\lambda_0>\lambda_1$. It follows from the same argument as in the proof of Lemma \ref{lambda1} that we can construct a supercritical branching process $Z_t$ that is dominated by $\xi^{\lambda_0}_t$. Moreover, there exists some constant $c>0$ and a non-negative random variable $W_0$ with $P(W_0>0)>0$ such that 
$$\frac{Z_t}{e^{ct}}\to W_0\quad \text{ as }t\to\infty.$$

To simplify notation we will write $\sigma=e^{-2}(1-e^{-\lambda_0})P(W_0>0)$. For a given $\delta>0$ we will choose $m$ so that $(1-\sigma)^m<\delta$. Since $\tau^\lambda_m<\infty$ almost surely on $\Omega_\infty\equiv\{\xi^\lambda_t\neq \varnothing \hh \forall t\geq 0\}$, we can choose $M_2$ sufficiently large so that 
\beq\label{error1}
P(\tau^{\lambda}_m\leq M_2)\geq (1-\delta)p(\lambda).
\eeq
When $\lambda'<\lambda$ is sufficiently close to $\lambda$ we also have 
\beq\label{error2}
P(\xi^{\lambda}_t \neq \xi^{\lambda'}_t \text{ for some }t\leq M_2)<\delta.
\eeq

Without loss of generality we can assume $\lambda'>\lambda_0$. Observe that each vertex in the frontier $\xi^{\lambda'}_{t} \cap F(A_{t})$ can start an independent branching process with probability $\geq e^{-2}(1-e^{-\lambda_0})$ and each branching process lives forever with probability at least $P(W_0>0)$ since $\lambda'>\lambda_0$. On the event $\{\xi^{\lambda}_t=\xi^{\lambda'}_t \text{ for }t\leq M_2\}\cap\{ \tau^{\lambda}_m\leq M_2\}$, the process $\xi^{\lambda'}_t$ can give birth to $X\sim Binomial(m, \sigma)$ branching processes that survive. If $X\geq 1$ then the process $\xi^{\lambda'}_t$ survives since it dominates a surviving branching process. Hence
\begin{align}\label{surv}
p(\lambda')
&\geq P\left(\{\xi^{\lambda}_t=\xi^{\lambda'}_t \text{ for }t\leq M_2\}\cap\{ \tau^{\lambda}_m\leq M_2\}\right)P(Binomial(m, \sigma)\geq 1)
\end{align}
It follows from (\ref{error1}) and (\ref{error2}) that when $\lambda'$ is sufficiently close to $\lambda$
$$P\left(\{\xi^{\lambda}_t=\xi^{\lambda'}_t \text{ for }t\leq M_2\}\cap\{ \tau^{\lambda}_m\leq M_2\}\right)\geq (1-\delta)p(\lambda)-\delta,$$
which then gives
$$
p(\lambda')\geq ((1-\delta)p(\lambda)-\delta)(1-(1-\sigma)^m)\geq ((1-\delta)p(\lambda)-\delta)(1-\delta).
$$
Therefore, for any given $\ep>0$ there exists $\delta$ sufficiently small and $\lambda'$ sufficiently close to $\lambda$ so that $p(\lambda')>p(\lambda)-\ep$. 
\end{proof}

\begin{remark}
The same construction can be carried out on periodic trees straightforwardly, thus proving the continuity of survival probability $p(\lambda)$ for the contact process on periodic trees. As a consequence, the contact process dies out at $\lambda_1$ on periodic trees.
\end{remark}



\section{No strong survival at $\lambda_2$} 
We start by introducing some new notation. Let $\GW^+(D)$ be the set of all realizations of the ``root-added" Galton-Watson tree in which the root 0 has offspring number 1 and its only offspring is the root of a Galton-Watson tree with offspring distribution $D$. A realization in $\GW^+(D)$ is denoted by $\TT^+$ and the measure over the realizations by $\nu$.

Define the measure $P_{\lambda, +}$ to be
$$P_{\lambda,+}(\hh\cdot\hh)=E_{\TT^+\sim \nu}(P_{\TT^+,\lambda}(\hh\cdot\hh))=\sum_{\TT^+ \in \GW^+(D)}P_{\TT^+,\lambda}(\hh\cdot\hh)\nu(\TT^+).$$
We will drop the index $\lambda$ when the context is clear. The contact process on $\GW^+(D)$ with with root 0 initially infected will be denoted by $\eta^0_t$ while the contact process on $\GW(D)$  will still be denoted by $\xi^0_t$.

From now on we set $A_t=\cup_{0\leq s\leq t}\hh \eta^0_s$. Let $|v|$ denote the distance from vertex $v$ to the root 0. For a vertex $x\in \eta^{0}_{t}\cap F(A_{t})$, define a \textit{branch} of $x$ to be
$$S^+(x)=\{x\}\cup S(x'),$$
where $x'$ can be any child of $x$ that has not been infected by time $t$. Observe that $S^+(x)$ follows the distribution of  $\GW^+(D)$. 

\begin{lemma}\label{expsize}
If $\lambda>\lambda_1$, there exists $\alpha,\beta,\delta>0$ and $t_1$ such that for all $j\geq 1$, 
$$P_{\lambda,+}( |\eta^{0}_{jt_1}\cap F(A_{jt_1})\cap \{v: j \alpha t_1\leq |v|\leq j \beta t_1\}|\geq e^{jc_\lambda t_1} \})\geq \delta,$$
where $c_\lambda$ is a positive number which may depend on $\lambda$.
\end{lemma}
\begin{proof}
Let $\theta=P_{+}(\eta^0_t\neq \varnothing \text{ for all }t\geq 0)$. It follows from the proof of Lemma \ref{BPcomparison} that there exists $c>0$ such that when $t_1$ is sufficiently large
$$P_{+}(|\eta^{0}_{t_1}\cap F(A_{t_1})|\geq e^{c t_1})\geq \frac{3}{4}\theta.$$
For reasons that will become clear later, we choose $t_1$ large enough so that the following also holds
\beq\label{chooset}
(\theta e^{ct_1})/8>1
\eeq
Given the choice of $t_1$ we can then choose $\alpha$ small and $\beta$ large so that 
\beq\label{good}
P_{+}(|\eta^{0}_{t_1}\cap F(A_{t_1})\cap\{v: \alpha t_1\leq |v|\leq \beta t_1\}|\geq e^{c t_1})\geq \frac{1}{2}\theta\equiv \delta.
\eeq

A vertex $x$ is said to be \textit{good} if the event in (\ref{good}) with probability $\delta$ occurs for the contact process on $S^+(x)$ with $x$ initially infected. If the root 0 is good, we say the root is a good vertex at level 0. The good vertices in $\eta^{0}_{jt_1}\cap F(A_{jt_1})\cap\{v: j\alpha t_1\leq |v|\leq j\beta t_1\} $ are called the good vertices at level $j$. Let $N_j$ be the number of good vertices at level $j$. Trivially we have 
$$|\eta^0_{jt_1}\cap F(A_{jt_1})\cap\{ v: j\alpha t_1\leq |v|\leq j\beta t_1\} |\geq N_j.$$
Hence it suffices to estimate $N_j$. Note that the good vertices in $\eta^{0}_{t_1}\cap F(A_{t_1})\cap \{v: \alpha t_1\leq |v| \leq \beta t_1\}$ correspond to disjoint branches, so $N_1$ dominates a $Binomial(e^{c t_1}, \delta)$ random variable if the root 0 is good. 

For a random variable $X\sim Binomial(n,p)$, Chebyshev's inequality gives
\beq\label{chebyshev}
P\left(X\leq \frac{np}{2}\right)\leq \frac{Var(X)}{ (np/2)^2}\leq \frac{4}{np}.
\eeq
By (\ref{chebyshev}) with $n=e^{ct_1}$ and $p=\delta$,
\begin{align*}
P\left(N_1\leq \frac{\delta e^{c t_1}}{2}\bigg| N_0=1\right)&\leq P\left(Binomial(e^{c t_1}, \delta)\leq  \frac{\delta e^{c t_1}}{2}\right)\\
&\leq (4/\delta)e^{-c t_1},
\end{align*}
where $(4/\delta)e^{-c t_1}<1$ by (\ref{chooset}). The branches of good vertices at level 1 are disjoint, so 
 the random variable $N_2$ dominates $Binomial(N_1 e^{c t_1}, \delta)$, which gives
\begin{align*}
P\left(N_2\leq \left(\frac{\delta e^{c t_1}}{2}\right)^2\bigg|N_1\geq \frac{\delta e^{c t_1}}{2}\right)& \leq P\left(Binomial\left(\frac{\delta e^{2c t_1}}{2},\delta\right)\leq  \left(\frac{\delta e^{c t_1}}{2}\right)^2\right)\\
&\leq 2\left(\frac{2e^{-ct_1}}{\delta}\right)^2\leq \left((4/\delta)e^{-ct_1}\right)^2.
\end{align*}

It follows from the Markov property that $N_{i+1}$ depends only on $N_i$ for $i\geq 0$. Hence, for any $j\geq 1$
\begin{align*}
P\left(N_j \geq  \left(\frac{\delta e^{c t_1}}{2}\right)^j \right)&\geq \delta P\left(N_1\geq \frac{\delta e^{ct_1}}{2}\right)\prod_{i=1}^{j-1} P\left( N_{i+1}\geq  \left(\frac{\delta e^{c t_1}}{2}\right)^{i+1} \bigg| N_{i}\geq  \left(\frac{\delta e^{c t_1}}{2}\right)^{i}\right) \\
&\geq \delta \prod_{i=1}^j \left(1-\left((4/\delta)e^{-ct_1}\right)^i\right )\geq \delta\sigma
\end{align*}
where $\sigma \equiv \prod_{i=1}^\infty \left(1-((4/\delta)e^{-ct_1})^i\right)>0$.
Since $(\theta e^{ct_1})/8>1$ there exists $0<c_\lambda<c$ such that 
$$e^{c_\lambda t_1}<2 < \frac{\theta e^{ct_1}}{4}= \frac{\delta e^{c t_1}}{2}.$$
The last equality follows from the definition $\delta=\theta/2$.
Hence for any $j\geq 1$, $P(N_j \geq  e^{jc_\lambda t_1})\geq \delta\sigma>0$.
\end{proof}
 
We are now ready to prove Theorem \ref{lambda2}, i.e., the contact process on $\GW(D)$ does not survive strongly at $\lambda_2(\GW(D))$.

\mn\textit{Proof of Theorem \ref{lambda2}.}
Since $P_{\lambda_2}(0\in \xi^0_t \text{ i.o.})=0$ if and only if $P_{\lambda_2,+}(0\in \eta^0_t \text{ i.o.})=0$, it suffices to show the latter. Our goal is to show that $P_{\lambda,+}(0\in \eta^0_t \text{ i.o.})>0$ implies $\lambda>\lambda_2$.
Now suppose $P_{\lambda, +}(0\in \eta^{0}_t \text{ i.o.})>0$. Theorem \ref{lambda1} implies that $\lambda>\lambda_1$, so there exists $\lambda_0$ such that $\lambda>\lambda_0>\lambda_1$. Applying Lemma \ref{expsize} to the contact process with infection rate $\lambda_0$, there exists positive $c_{\lambda_0}$ so that  for all $j\geq 1$,
\beq\label{bj}
P_{\lambda_0,+}(|\eta^{0}_{jt_1}\cap F(A_{jt_1})\cap \{v: j\alpha t_1\leq |v|\leq j\beta t_1\}|\geq e^{jc_{\lambda_0} t_1})\geq \delta
\eeq
for some choice of $\alpha, \beta,\delta$ and $t_1$.

Let $B_j=\eta^0_{jt_1}\cap F(A_{jt_1})\cap\{v: j\alpha t_1\leq |v|\leq j\beta t_1\}$ and let $x$ be a vertex in $B_j$. We use the notation $(x,0)\to (0,t)$ to denote the event that running the contact process on the shortest path connecting $x$ and $0$ with $x$ initially infected, $0$ will be infected at time $t$. Trivially, if every particle in the direct path from $x\in B_j$ to 0 lives for at least $2/\alpha$ and pushes the infection toward the root 0 between time $[1/\alpha, 2/\alpha]$ then
$$P_{\lambda_0}( (x,0)\to (0,t) \text{ for some $t\in[jt_1, (2\beta/\alpha)j t_1]$})\geq \big(e^{-2/\alpha}(e^{-\lambda_0/\alpha}-e^{-2\lambda_0/\alpha})\big)^{j\beta  t_1}\equiv (q_0)^{j\beta t_1}.$$
Fix $\ep=(c_{\lambda_0}/2)$ and $\rho=\frac{1}{4}P_{\lambda,+}(0\in \eta^0_t \text{ i.o.})$. For reasons that will become clear later we choose $j_0$ such that 
\beq\label{choosej}
C_{j_0}\equiv \delta \rho  e^{(c_{\lambda_0}-\varepsilon) j_0 t_1}>1.
\eeq
From now on we write $t_0\equiv j_0t_1$. We want to use the vertices in $B_{j_0}$ to push the infection back to the root 0. For some vertex $x\in B_{j_0}$, let $\eta^x_t$ denote the contact process restricted to the branch $S^+(x)$ starting with $x$ infected. If $x\in \eta^x_t$ for many times, then the infection can try many times to push back to the root. Hence the probability of successfully pushing the infection from $x$ back to 0 is increased if $x\in \eta^x_t$ for many times. To estimate the success probability let $\sigma^{x}_0=0$ and define $$\sigma^{x}_{i+1}=\inf\{ t>\sigma_i+(2\beta/\alpha)t_0: x\in \eta^x_t\}.$$
Choose $m$ so that 
\beq\label{push}
1-(1-(q_0)^{\beta t_0})^m \geq e^{-\ep t_0},
\eeq
which is a lower bound on the probability that $x$ successfully pushes the infection to the root 0 within $m$ independent trials. If the process $\eta^x_t$ restricted to $S^+(x)$ reinfects $x$ for infinitely many times then $\sigma^x_m<\infty$, i.e.,
$P_{\lambda,+}(\sigma^x_m<\infty) \geq P_{\lambda,+}(0\in \eta^0_t \text{ i.o.})$.
So we can choose $M_2$ such that 
\beq\label{recurrent}
P_{\lambda,+}(\sigma^x_m\leq M_2) \geq \frac{1}{2}P_{\lambda,+}(0\in \eta^0_t \text{ i.o.})= 2\rho.
\eeq
For any $\epsilon<\lambda-\lambda_0$, the estimates in (\ref{bj}) through (\ref{push}) are still true for the contact process with infection rate $\lambda-\epsilon$ since $\lambda-\epsilon>\lambda_0$. By (\ref{recurrent}) there exists a sufficiently small $\epsilon>0$ with $\epsilon<\lambda-\lambda_0$ that satisfies
$$P_{\lambda-\epsilon,+}(\sigma^x_m\leq M_2) \geq \rho.$$

A vertex $x$ is said to be \textit{$(m,\lambda-\epsilon)$-recurrent} if $\{\sigma^x_m\leq M_2\}$ occurs for the contact process $\eta^{x}_t$ with infection rate $\lambda-\epsilon$. If $x$ is $(m,\lambda-\epsilon)$-recurrent then with probability $\geq e^{-\varepsilon t_0}$ it can reinfect the root after time $t_0$. Each vertex in $B_{j_0}$ is $(m,\lambda-\epsilon)$-recurrent independently with probability at least $\rho$.

Now we are ready to show $\lambda>\lambda_2$. The following argument is inspired by the proof of Proposition 4.57 in Liggett \cite{Liggett}. Let $r_i=P_{\lambda-\epsilon,+}(\exists\hh t\geq 2it_0 \text{ such that }0\in\eta^0_t)$. To estimate $r_{i+1}$ we first run the contact process up to time $t_0$. According to (\ref{bj}) at time $t_0$ the set of infected frontier $B_{j_0}$ has size $\geq L_0\equiv \exp(c_{\lambda_0}t_0)$ with probability $\delta$. Recall that $\eta^x_t$ is the contact process restricted to the branch $S^+(x)$ starting with $x$ infected. Let 
$$R_x=\{ \exists\hh t\geq 2it_0 \text{ such that }x\in \eta^x_t\} \quad \text{ and }\quad G_x=\{ x \text{ is $(m,\lambda-\epsilon)$-recurrent}\}.$$ 
If for some $x\in B_{j_0}$ the event $R_x\cap G_x$ occurs then $x$ is reinfected after $(2i+1)t_0$ and can push the infection back to the root after $2(i+1)t_0$ with probability at least $e^{-\varepsilon t_0}$. It follows that 
\begin{align}\label{ri}
r_{i+1}
\nonumber &\geq  P_{\lambda-\epsilon,+}\left(\exists x\in B_{j_0} \text{ such that } R_x\cap G_x \text{ occurs}\right)\cdot e^{-\varepsilon t_0}\\
&\geq P_{\lambda_0,+}\big( Binomial(|B_{j_0}|, P_{\lambda_0,+}(R_x\cap G_x))\geq 1\big)\cdot e^{-\varepsilon t_0} \quad\quad \text{(since $\lambda_0<\lambda-\epsilon$).}  
\end{align}
On a given tree $\TT^+$ $R_x$ and $G_x$ are positively correlated, as they are both increasing functions of the infection events and decreasing functions of the recovery events (see Liggett \cite{Liggett} for an account of positive correlation). Thus
$$P_+(R_x\cap G_x)=E_{\TT^+\sim \nu}( P_{\TT^+}(R_x\cap G_x))\geq E_{\TT^+\sim \nu}( P_{\TT^+}(R_x)\cdot P_{\TT^+}(G_x)).$$
In order to estimate $E_{\TT^+\sim \nu}( P_{\TT^+}(R_x)\cdot P_{\TT^+}(G_x))$ we will adopt a normal positive correlation argument. The reader can skip this part and jump straight to (\ref{lowerb}).

Now observe that the root-added Galton-Watson tree $\TT^+$ can be represented by a collection of i.i.d. random variables. Let the random variables $Z_0=1$ and $Z_1\sim D$ be the number of offspring of the added root and the root, respectively. Then we sample a series of i.i.d. random variables $Z_{1,1},\dots, Z_{1,Z_1} \sim D$ to represent the number of offspring for each individual in the first generation. The offspring of each individual in the second generation is then given by the collection $\{Z_{1,i,1}, \dots, Z_{1,i ,Z_{1,i}}\}_{i=1}^{Z_{1}}$, and so on. 

To simplify notation we write the index $(i_1,i_2,\dots,i_n)$ as $\mathbf{i}=(i_1,i_2,\dots,i_n, 0,\dots)\in \mathbb{N}^{\mathbb{N}}$. Hence, given a collection $\vec{Z}=\{Z_{\mathbf{i}}:  \mathbf{i}\in \mathbb{N}^{\mathbb{N}} \}$ of i.i.d. random variables we can construct a Galton-Watson tree accordingly. Note that different $\vec{Z}$'s can be mapped to the same $\TT^+\in\GW^+(D)$. For example, if $Z_1=1$ then we only use the element $Z_{1,1}$ to construct the next generation, ignoring the elements $Z_{1,2}, Z_{1,3}$ and so on. Let $\mu^*$ denote the measure over the realizations of $\vec{Z}$. A function $h$ of $\vec{Z}$ is said to be increasing if for a pair $\vec{Z}^1, \vec{Z}^2$ with $Z^1_{\mathbf{i}}\leq Z^2_{\mathbf{i}}$ for all $\mathbf{i}$ we have $h(\vec{Z}^1)\leq h(\vec{Z}^2)$. It follows from the Harris inequality (see Section 2.2 in \cite{Grimmett}) that $\mu^*$ has positive correlations, that is, $E_{\mu^*}fg\geq E_{\mu^*}f \cdot E_{\mu^*}g$ for all increasing functions $f$ and $g$ with finite second moments. Note that the measure $\nu$ on $\GW^+(D)$ can be obtained from $\mu^*$ by collapsing different $\vec{Z}$'s mapping to the same $\TT^+$, which means, for example,
$$E_{\TT^+\sim \nu}( P_{\TT^+}(R_x))=E_{\vec{Z}\sim \mu^*}(P_{\vec{Z}}(R_x))$$
Both $P_{\vec{Z}}(R_x)$ and $P_{\vec{Z}}(G_x)$ are increasing with respect to $\vec{Z}$, so it follows from the positive correlations of $\mu^*$ that 
\begin{align*}
E_{\TT^+\sim \nu}\left( P_{\TT^+}(R_x)\cdot P_{\TT^+}(G_x)\right)&=E_{\vec{Z}\sim \mu^*}( P_{\vec{Z}}(R_x)\cdot P_{\vec{Z}}(G_x))\geq E_{\vec{Z}\sim \mu^*}( P_{\vec{Z}}(R_x))\cdot E_{\vec{Z}\sim \mu^*}( P_{\vec{Z}}(G_x))\\ 
&=E_{\TT^+\sim \nu}( P_{\TT^+}(R_x))E_{\TT^+\sim \nu}( P_{\TT^+}(G_x))=P_+(R_x)P_+(G_x)
\end{align*}
That is,
\beq\label{lowerb}
P_+(R_x\cap G_x)\geq P_+(R_x)P_+(G_x)\geq r_i \rho.
\eeq

Now we can finish the calculation in (\ref{ri})
\begin{align*}
r_{i+1}\geq P(|B_{j_0}|\geq L_0)P(Binomial(L_0, r_i \rho)\geq 1) e^{-\varepsilon t_0}\geq \delta e^{-\varepsilon t_0} (1-(1-r_i \rho)^{L_0})
\end{align*}
where the term $\delta$ comes from (\ref{bj}).
Set $f(r)= \delta e^{-\varepsilon t_0}(1-(1-r\rho)^{L_0})$. Then $f(0)=0$ and $f(1)<1$. Now the reason for our choice of $j_0$ in (\ref{choosej}) has become clear. Under this choice of $j_0$ we have 
$$f'(0)=\delta \rho  e^{-\varepsilon t_0}  L_0=C_{j_0}>1.$$
It follows that $f$ has a fixed point $r^*\in (0,1)$, i.e., $f(r^*)=r^*$. Now we prove $r_i\geq r^*$ inductively for $i\geq 0$. First we have $r_0=1$. Suppose $r_i\geq r^*$, then the monotonicity of $f$ implies that 
$$r_{i+1}\geq f(r_i)\geq f(r^*)=r^*.$$
Thus we have proved $r_i\geq r^*$ for all $i\geq 0$, which implies that $P_{\lambda-\epsilon,+}(0\in \eta^0_t \text{ i.o.})>0$ and thus $\lambda-\epsilon\geq \lambda_2$. Therefore we have shown $\lambda>\lambda_2$ and this completes the proof. \qed

\section{The weak survival phase on periodic trees}
In this section we consider the contact process on a general periodic tree $\TT_\kappa$. We will prove (i) $\lambda_1<\lambda_2$ (see Theorem \ref{intphase}) and (ii) the contact process does not survive strongly at $\lambda_2$ (see Theorem \ref{pgrowth}, (ii)). The rest of Theorem \ref{pgrowth} can be proved by similar arguments to that of the Galton-Watson trees and hence is omitted here. The reader is referred to the remarks in previous sections for ideas of the proofs.

Recall that the periodic tree $\TT_\kappa=(a_1,a_2,\dots,a_\kappa)$ is arranged in the way that every vertex $x$ in $\TT_\kappa$ has one neighbor above it and $d(x)$ neighbors below it, where $d(x)\in\{a_1,\dots,a_\kappa\}$ is the offspring number of $x$. Throughout the discussion we write $\gamma\equiv (a_1\cdot a_2\cdots a_\kappa)$. 

Without loss of generality, a distinguished vertex $o$ is chosen to be the root and the root is said to be on level 0. Now each vertex can be assigned a \textit{level} according to their position relative to the root $o$. Specifically, we will define the function $\ell:\TT_\kappa\to \ZZ$ so that (i) $\ell(o)=0$ and (ii) for each $x\in \TT_\kappa$, $\ell(y)=\ell(x)-1$ for exactly one neighbor $y$ of $x$, and $\ell(y)=\ell(x)+1$ for the other $d(x)$ neighbors of $x$. For a vertex $x\in\TT_\kappa$, $\ell(x)$ is said to be the level of $x$.

\subsection{The growth profile}
We now follow Liggett \cite{Lig96} and Stacey \cite{Stacey} to define the weight of a vertex $x$ by 
$$w_\rho(x)=\rho^{\ell(x)},$$
where $\rho>0$ is a constant to be specified later. The weight of a set $A$ of vertices is defined to be 
$$w_\rho(A)=\sum_{x\in A}\rho^{\ell(x)}.$$

Take $\{e_n: n\in\ZZ\}$ in $\TT_\kappa$ such that $\ell(e_n)=n$ and $|e_n-e_{n+1}|=1$, where $e_0$ is what we called the root $o$. As many of the things we will discuss are analogous to results in Liggett \cite{Liggett} for the homogeneous tree $\TT^d$, we will use the same notation as in \cite{Liggett}. Throughout the discussion we will use $\xi_t$ to denote the contact process on $\TT_\kappa$ starting with only the root $o$ infected. 

Define
$$u(n)=P( e_n\in \xi_t \text{ for some }t).$$
In order for the infection $\xi_t$ to reach $e_{(n+m)\kappa}$ it has to first reach $e_{n\kappa}$. Thus if we restart the contact process with only $e_{n\kappa}$ infected at the moment $\xi_t$ reaches $e_{n\kappa}$, by the strong Markov property
\beq\label{subadd}
u((m+n)\kappa)\geq u(m\kappa)u(n\kappa).
\eeq
See (4.47) in Liggett \cite{Liggett} for a detailed argument. Hence 
$$\beta(\lambda):=\lim_{n\to\infty} [u(n\kappa)]^{1/n}$$
exists and satisfies 
$$\beta(\lambda)=\sup_n [u(n\kappa)]^{1/n}.$$
It is easy to show that $\beta(\lambda)$ is independent of the type of vertex that is initially infected. 



Let $S(x)$ denote the subtree that contains vertex $x$ and all of its descendants. We will use $\bar{\xi}_t$ to denote the contact process restricted to $S(o)$, starting with only the root $o$ infected. The following can be derived from the proof of Lemma 4.53 in \cite{Liggett}. Hence we only give a sketch of proof here.
\begin{lemma}\label{utbeta}
For any $\lambda$,
$$\lim_{n\to\infty} \left( \sup_t P(e_{n\kappa}\in \bar{\xi}_{t}) \right)^{\frac{1}{n}}=\beta(\lambda).$$
\end{lemma}
\mn
\textit{Sketch of proof.}
Let $B(x,n)=\{ y\in \TT_\kappa: |y-x|\leq n\}$. Let $v_r(m)$ be the probability that $\xi^o_t$ restricted to $B(o,r)$ infects $e_m$ for some $t\leq r$. Then 
$$\lim_{r\to\infty} v_r(m)=u(m).$$
Now take positive integers $j,k,m,n$ satisfying
$$jm+r\leq n.$$
Let $x_0=e_{(n-jm)\kappa}, x_1=e_{(n-(j-1)m)\kappa}, \dots, x_j=e_{n\kappa}$. In order to have an infection path from $(o,0)$ to $(e_{n\kappa},t)$ for some $t\leq nr$, we consider the intersection of following events: 
\begin{align*}
&\text{ the infection starting from $o$ reaches $x_0$ within time 1 without exiting $S(o)$},\\
&\text{ the infection from $x_0$ reaches $x_1$ within time $r$ without exiting $B(x_1,r)$},\\
& \dots \\
&\text{ the infection from $x_{j-1}$ reaches $x_j$ within time $r$ without exiting $B(x_j,r)$}.
\end{align*} 
Therefore
\begin{align*}
P(e_{n\kappa}\in \bar{\xi}_t \text{ for some }t&\leq nr )\geq P(e_{(n-jm)\kappa}\in \bar{\xi}_t \text{ for some }t\leq 1) [v_r(m\kappa)]^j
\end{align*}
Observe that 
$$P(e_{n\kappa}\in\bar{\xi}_{i+1})\geq e^{-1}P(e_{n\kappa}\in \bar{\xi}_t \text{ for some }t\in[i,i+1]).$$
Hence
\begin{align*}
\sup_t P(e_{n\kappa}\in \bar{\xi}_t)&\geq \max_{0\leq i<nr}P(e_{n\kappa}\in\bar{\xi}_{i+1}) \\
&\geq e^{-1} \max_{0\leq i<nr}P(e_{n\kappa}\in \bar{\xi}_t \text{ for some }t\in[i,i+1])\\
& \geq \frac{e^{-1}}{nr}P(e_{n\kappa}\in \bar{\xi}_t \text{ for some }t\leq nr )\\
&\geq  \frac{e^{-1}}{nr}P(e_{(n-jm)\kappa}\in \bar{\xi}_t \text{ for some }t\leq 1) [v_r(m\kappa)]^j.
\end{align*}
Therefore
$$\liminf_{n\to\infty}\left( \sup_t P(e_{n\kappa}\in \bar{\xi}_{t}) \right)^{\frac{1}{n}}\geq [v_r(m\kappa)]^{1/m}.$$
Taking $r\to\infty$ and then $m\to\infty$ gives the result. \qed

To simplify notation, in this section we will write $\lambda_1$ for $\lambda_1(\TT_\kappa)$ and $\lambda_2$ for $\lambda_2(\TT_\kappa)$. The structure of the periodic trees are similar in a way to the homogeneous trees, which allows us to obtain many useful properties on $\beta(\lambda)$. Here we only list one property to be used later:
\begin{lemma}\label{beta}
\mn
$\left(\frac{\lambda}{1+\lambda}\right)^\kappa\leq \beta(\lambda_1)\leq 1/\gamma$.
\end{lemma}
\begin{proof}
When the vertex $e_i$ is infected, the probability that it infects $e_{i+1}$ before the infection recovers is $\lambda/(1+\lambda)$. If this event occurs for every vertex on the path from $e_0$ to $e_{n\kappa}$ then $e_{n\kappa}\in \xi^o_t$ for some $t$. That is,
$$u(n\kappa)=P( e_n\in \xi_t \text{ for some }t)\geq \left(\frac{\lambda}{1+\lambda}\right)^{n\kappa},$$
which implies that $\beta(\lambda)\geq \left(\frac{\lambda}{1+\lambda}\right)^\kappa$.

Suppose $\beta(\lambda)>1/\gamma$. There exists some constant $a$ so that $\beta(\lambda)> a>1/\gamma$. By Lemma \ref{utbeta} there exists $t_0$ and $n_0$ such that 
$$P(e_{n_0\kappa}\in \bar{\xi}_{t_0})^{1/n_0}\geq a.$$
For $n\in\ZZ$, we will use
$$\LL_{n}=\{ x\in \TT_\kappa: \ell(x)=n\}$$
to denote the set of vertices on level $n$. 
Let $B_1=\{ x\in \bar{\xi}_{t_0}: x\in \LL_{n_0\kappa}\}$. Then $E|B_1|\geq (a\gamma)^{n_0}>1$. At time $t_0$, at each $x\in B_1$ we can start an independent contact process restricted to $S(x)$ and run them for time $t_0$. Continuing this construction gives a supercritical branching process. Therefore $\beta(\lambda)>1/\gamma$ implies $\lambda>\lambda_1$, i.e., $\beta(\lambda_1)\leq 1/\gamma$. 

\end{proof}

\subsection{Proof of Theorem \ref{intphase}}
The above discussion completes our preparation for the proof. Now we will proceed to show the existence of an intermediate phase for the contact process on periodic trees. 

There are $\kappa$ types of vertices in the periodic tree $T_\kappa$ categorized according to their offspring numbers and relative positions in the tree. A vertex that corresponds to $a_{i+1}$ in the tree $T_\kappa=(a_1, a_2, \dots, a_\kappa)$ is said to be of type $i$, where $i\in\{0,1,\dots,\kappa-1\}$. Let $\xi^{i,\lambda}_t$ be the contact process on $T_\kappa$ with infection rate $\lambda$, starting with an infected type $i$ vertex that we denote by $x_i$. We will also set $x_i$ to be the root on level 0 and define the function $\ell: T_{\kappa}\to \ZZ$ accordingly. 

If we only look at levels $\{\LL_{n\kappa}\}_{n\in\ZZ}$ of the tree $\TT_{\kappa}$ then the structure resembles a homogeneous tree $\TT^{\gamma}$. Here we state a Lemma in Liggett \cite{Liggett} that will be useful in our proof.

\mn \textbf{Lemma 4.26. (See \cite{Liggett} Part I)}
\textit{ Let $T^d$ denote the homogeneous tree with degree $d+1$ and root $o$. Let 
$$\alpha_n(\rho)=\sum_{|x-o|=n} \rho^{\ell(x)}.$$
Then $\alpha_0(\rho)=1$ and for $n\geq 1$
$$\alpha_{n}(\rho)=
\begin{cases}
\frac{d^{n-1}\rho^n(d^2\rho^2-1)+\rho^{-n}(\rho^2-1)}{d\rho^2-1}, &d\rho^2\neq 1\\
\rho^{-n}[ (n+1)-(n-1)\rho^{2}], & d\rho^2=1.
\end{cases}
$$
}

\begin{lemma}\label{weight}
For any $\varep>0$, there exists $t_0>0$ so that 
$$\max_{i\in \{0,1,\dots,\kappa-1\}} E w_\rho(\xi^{i,\lambda_1}_{t_0})\leq \varep.$$
\end{lemma}

\begin{proof}
We will consider the contact process $\xi^{i,\lambda_1}_t$ starting with a type $i$ vertex $x_i$ infected and set $x_i$ to be the root of $T_\kappa$. Recall that $\LL_{n}=\{ x\in \TT_\kappa: \ell(x)=n\}$ denotes the set of vertices on level $n$ and $S(x)$ denotes the subtree that contains vertex $x$ and all of its descendants. For a vertex $x\in\cup_{n\in \ZZ}\LL_{n\kappa}$ we shall define a block with root $x$ to be
$$S^x_\kappa =\{ y\in S(x): |y-x| \leq \kappa-1\}.$$
Note that every vertex in $\cup_{n\in \ZZ}\LL_{n\kappa}$ has type $i$. Define $\mathcal{D}_n(x_i)=\{ x\in T_{\kappa}: |x-x_i|= n\}$ and 
$$\mathcal{B}_n(x_i)=\{ x\in T_\kappa : |x-x_i|=n\kappa, \ell(x)\mod\kappa =0\}.$$
The tree $T_\kappa$ can be partitioned into a disjoint union of blocks 
$$T_\kappa=\cup_{x\in \cup_{m\in\ZZ}\mathcal{L}_{m\kappa}} S^x_\kappa.$$
The weight $Ew_\rho(\xi^{i,\lambda_1}_t)$ can be split into two parts:
\beq\label{twopart}
Ew_\rho(\xi^{i,\lambda_1}_t)
\leq  \sum_{n=0}^{\kappa N-1}\sum_{x\in \mathcal{D}_{n}(x_i)} w_\rho(x)P(x\in \xi^{i,\lambda_1}_{t})+\sum_{n= \kappa N}^{\infty}\sum_{x\in \mathcal{D}_{n}(x_i)} w_\rho(x)P(x\in \xi^{i,\lambda_1}_{t})
\eeq
To obtain an upper bound on the second term in (\ref{twopart}), we count the weight of the whole block $S^x_\kappa$ if some vertex in this block is infected at time $t$. Hence
\beq\label{block}
\sum_{n= \kappa N}^{\infty}\sum_{x\in \mathcal{D}_{n}(x_i)} w_\rho(x)P(x\in \xi^{i,\lambda_1}_{t})
\leq \sum_{n\geq N}\sum_{x\in \mathcal{B}_n(x_i)} w_\rho(S^x_\kappa) P( S^x_\kappa \cap \xi^{i,\lambda_1}_{t}\neq \varnothing)
\eeq
Set $a_0=1$ and it is straightforward to compute
\beq\label{first}
w_\rho(S^x_\kappa)=\sum_{m=0}^{\kappa-1}\left( \rho^{\ell(x)+m}\cdot \prod_{j=0}^m a_j \right) \leq \gamma \rho^{\ell(x)} \left(\sum_{m=0}^{\kappa-1} \rho^m \right) \leq \frac{\gamma}{1-\rho} \cdot \rho^{\ell(x)}.
\eeq
Let $\Gamma_{x_i,x}$ denote the path on $T_\kappa$ from $x_i$ to $x$. If $\Gamma_{x_i,x}\cap S^x_\kappa\neq\{x\}$, then the first vertex in $S^x_{\kappa}$ that will be reached by the infection is not $x$.  In this case, let $x^*\in \Gamma_{x_i,x}$ denote the vertex that satisfies $|x-x^*|=\kappa$. That is, the infection starting at $x_i$ would spread to $S^x_\kappa$ through vertex $x^*\in S^x_\kappa$. Then we know
$$P(S^x_\kappa \cap \xi^{i,\lambda_1}_t \neq \varnothing)\leq
\begin{cases}
P(x \in \xi^{i,\lambda_1}_t \text{ for some }t) & \text{ if $\Gamma_{x_i,x}\cap S^x_\kappa=\{x\}$,}\\
P(x^* \in \xi^{i,\lambda_1}_t \text{ for some }t) & \text{ otherwise.}
\end{cases}
$$
Therefore, for $x\in \mathcal{B}^i_n$,
\beq\label{second}
P(S^x_\kappa \cap \xi^{i,\lambda_1}_t \neq \varnothing)\leq \beta(\lambda_1)^{|x_i-x|-1}=\beta(\lambda_1)^{n-1}.
\eeq
Putting (\ref{first}) and (\ref{second}) together we have 
\beq\label{twopart2}
(\ref{block})\leq \sum_{n\geq N}\sum_{x\in \mathcal{B}_n(x_i)}  \frac{\gamma}{1-\rho}\cdot \rho^{\ell(x)} \beta(\lambda_1)^{n-1}.
\eeq
Define $\alpha_n(\rho)=\sum_{x\in \mathcal{B}_n(x_i)}  \rho^{\ell(x)}$. As a simple corollary of Lemma 4.26 we have $\alpha_0(\rho)=1$ and
\beq\label{an}
\alpha_{n}(\rho)=\rho^{-n\kappa}[ (n+1)-(n-1)\rho^{2\kappa}] \quad\text{ when } \gamma\rho^{2\kappa}=1.
\eeq
It follows from (\ref{an}) and $\beta(\lambda_1)\leq 1/\gamma$ (see Lemma \ref{beta}) that we can choose $N$ sufficiently large so that 
\beq\label{part1}
\sum_{n\geq N}\sum_{x\in \mathcal{B}_n(x_i)}  \frac{\gamma}{1-\rho}\cdot \rho^{\ell(x)} \beta(\lambda_1)^{n-1}
\leq \frac{\gamma}{\beta(\lambda_1)(1-\rho)}  \left(\sum_{n\geq N} \gamma^{-n/2}(n+1)\right)\leq \varep/2.
\eeq
Notice that this choice of $N$ is independent of the type of vertex that is initially infected as well as the time $t$. When $N$ is fixed, since at $\lambda_1$ the contact process dies out
$$\lim_{t\to\infty} \sum_{n=0}^{\kappa N-1}\sum_{x\in \mathcal{D}^i_{n}} w_\rho(x)P(x\in \xi^{i,\lambda_1}_{t}) =0$$
for any $i\in \{0,1,\dots, \kappa-1\}$. That is, there exists $M_i>0$ so that when $t\geq M_i$,
\beq\label{part2}
\sum_{n=0}^{\kappa N-1}\sum_{x\in \mathcal{D}^i_{n}} w_\rho(x)P(x\in \xi^{i,\lambda_1}_{t})\leq \varep/2.
\eeq
Taking $t_0=\max\{M_0, M_1, \dots, M_{\kappa-1}\}$ and adding up (\ref{part1}) and (\ref{part2}) proves the desired result.
\end{proof}

\mn
\textit{Proof of Theorem \ref{intphase}.} 
Given $\varep<1/4$, by Lemma \ref{weight} we can choose $t_0$ so that 
$$\max_{i\in\{0,1,...,\kappa-1\}}E_\lambda w_\rho(\xi^{i,\lambda_1}_{t_0})\leq \varep.$$
Since $\max_{i\in\{0,1,...,\kappa-1\}}E_\lambda w_\rho(\xi^{i,\lambda}_{t_0})$ is a continuous function with respect to $\lambda$, there exists $\lambda>\lambda_1$ so that 
$$\max_{i\in\{0,1,...,\kappa-1\}}E_\lambda w_\rho(\xi^{i,\lambda}_{t_0})<2\varep\equiv \delta.$$
Without loss of generality we consider the contact process $\xi^{0,\lambda}_t$ that starts with a type 0 vertex infected. Let $\FF_t$ be the $\sigma$-algebra generated by $\xi^{0,\lambda}_t$ up to time $t$. We can show that
\begin{align*}
E[w_\rho(\xi^{0,\lambda}_{(n+1)t_0})|\FF_{nt_0}]&\leq \sum_{x\in \xi^{0,\lambda}_{nt_0}} \rho^{\ell(x)} E \sum_{y\in \xi^{x,\lambda}_{t_0}} \rho^{\ell(y)-\ell(x)}\\
&\leq \left(\sum_{x\in \xi^{0,\lambda}_{nt_0}} \rho^{\ell(x)}\right) \cdot \max_{i\in\{0,1,...,\kappa-1\}}E w_\rho(\xi^{i,\lambda}_{t_0})\leq \delta \cdot w_\rho(\xi^{0,\lambda}_{nt_0}).
\end{align*}
It follows that 
$$M_n:=\frac{w_\rho(\xi^{0,\lambda}_{nt_0})}{\delta^n}$$
is a non-negative supermartingale, which converges almost surely as $n$ goes to infinity. Since $\delta=2\varep<1/2$,  
$$w_\rho(\xi^{0,\lambda}_{nt_0})\to 0 \quad\quad \text{ as }n\to\infty.$$
Intuitively this implies that the process $\xi^{0,\lambda}_t$ (where $\lambda>\lambda_1$) does not survive strongly. See Proposition 1.0 in \cite{Stacey} for a complete argument. Therefore $\lambda_1<\lambda\leq \lambda_2$.\qed

\subsection{Weak survival at $\lambda_2$}\label{pweak}
In this section we complete the proof of Theorem \ref{pgrowth} by showing the contact process on $\TT_\kappa$ does not survive strongly at $\lambda_2$. Since $\lambda_1(\TT_\kappa)<\lambda_2(\TT_\kappa)$ we know the process survives weakly at $\lambda_2$. The proof on periodic trees is very much similar to that of Galton-Watson trees. In this case, the estimations are simpler due to the more regular structure that comes from periodicity.

Without loss of generality, a vertex with offspring number $a_1$ is chosen to be the root $o$. Let $\bar{\xi}_t$ denote the contact process restricted to the subtree $S(o)$ starting with the root $o$ infected.

\mn\textit{Proof of Theorem \ref{pgrowth} (ii).} Our goal is to show that $\rho\equiv P(o\in \bar{\xi}^{\lambda}_t \text{ i.o.})>0$ implies $\lambda>\lambda_2$. 

Suppose when the infection rate is $\lambda$ we have $\rho>0$. First observe that 
$$\rho=P(o \in \bar{\xi}^{\lambda}_t \text{ for a sequence of times }t\uparrow \infty)\leq u(n)$$
for all $n$ and $\lambda>0$. (The reader is referred to Theorem 4.65 in \cite{Liggett} for more details.) Hence 
$\beta(\lambda)=\lim_{n\to\infty} u(n)^{1/n}\geq \lim_{n\to\infty}\rho^{1/n}=1.$
Lemma \ref{utbeta} then implies that for any $\delta>0$ there exists $n_0$ and $t_0$ so that 
$$P(e_{n_0\kappa}\in \bar{\xi}^{\lambda}_{t_0})\geq (1-\delta)^{n_0}.$$
As $P(e_{n_0\kappa}\in \bar{\xi}^{\lambda}_{t_0})$ is continuous with respect to $\lambda$, there exists $\varep>0$ so that 
$$P(e_{n_0\kappa}\in \bar{\xi}^{\lambda-\varep}_{t_0})\geq (1-2\delta)^{n_0}.$$

Now we will estimate the size of the infection by comparing with a branching process. Define $\mathcal{L}_{n}=\{ x\in S(o): |x-o|=n\kappa\}$. Let $B_0=\{o\}$ and  $B_1=\{ x\in \bar{\xi}^{\lambda-\varep}_{t_0}: x\in \mathcal{L}_{n_0}\}$. Then $E|B_1|\geq (\gamma(1-2\delta))^{n_0}>1$ when we choose $\delta$ to be sufficiently small. Specifically we will choose $\delta$ so that $\gamma(1-2\delta)^2>1$ for reasons that will become clear later.  

At time $t_0$, at each $x\in B_1$ we can start an independent contact process restricted to $S(x)$ and run it for time $t_0$. Continuing this construction gives a supercritical branching process $|B_m|$ with mean $E|B_1|$. It is clear from the construction that $B_m\subset \bar{\xi}^{\lambda-\varep}_{mt_0}\cap \LL_{mn_0}$. Since $|B_m|$ is a supercritical branching process,  
$$\lim_{m\to\infty} \frac{|B_m|}{ (E|B_1|)^m}$$
exists almost surely and is not trivially zero. Therefore, there exists a $\epsilon>0$ such that 
\beq\label{bsize}
P(|B_m|\geq \epsilon(E|B_1|)^m)\geq \epsilon
\eeq
for all sufficiently large $m$.
For reasons that will become clear later we choose $m_0$ sufficiently large so that 
$$P(|\bar{\xi}^{\lambda-\varep}_{m_0t_0}\cap \LL_{m_0n_0}|\geq \epsilon(\gamma(1-2\delta))^{m_0n_0})\geq P(|B_{m_0}|\geq \epsilon(\gamma(1-2\delta))^{m_0n_0})\geq \epsilon$$
and
$$\epsilon^2 ((1-2\delta)^2\gamma)^{m_0n_0}>1.$$
Let $r_i=P(o\in\bar{\xi}^{\lambda-\varep}_{2im_0t_0})$ and $M\equiv \epsilon(\gamma(1-2\delta))^{m_0}$. We can obtain the following recursive relation:
\begin{align}\label{recursive}
\nonumber r_{i+1}&\geq P( Binomial(|\bar{\xi}^{\lambda-\varep}_{m_0t_0}\cap \mathcal{L}_{m_0n_0}|, r_i)\geq 1)P(e_{m_0n_0\kappa}\in \bar{\xi}^{\lambda-\varep}_{m_0t_0})\\
&\geq P(|\bar{\xi}^{\lambda-\varep}_{m_0t_0}\cap \mathcal{L}_{m_0n_0}|\geq M)(1-(1-r_i)^M)P(e_{n_0\kappa}\in \bar{\xi}^{\lambda-\varep}_{t_0})^{m_0}\\
&\geq \epsilon (1-2\delta)^{m_0n_0}(1-(1-r_i)^M).
\end{align}
Setting $f(r)=\epsilon (1-2\delta)^{m_0n_0}(1-(1-r)^M)$. By our choices of $\delta$ and $m_0$,
$$f'(0)=\epsilon^2 ((1-2\delta)^2\gamma)^{m_0n_0}>1.$$
Following the same argument as in the proof of Theorem \ref{lambda2} we can show there exists $r^*>0$ such that $r_i\geq r^*$ for all $i\geq 0$. That is, $\bar{\xi}^{\lambda-\varep}_t$ survives strongly, which implies $\lambda>\lambda-\varep\geq \lambda_2$. \qed

\end{document}